\documentclass[11pt]{article}
\usepackage[T1]{fontenc}
\usepackage{lmodern}
\usepackage[a4paper,margin=28mm]{geometry}
\usepackage{amsmath,amssymb,amsthm,mathtools}
\usepackage{booktabs,array,enumitem,microtype}
\usepackage[hidelinks]{hyperref}
\hypersetup{pdftitle={A Proof of the Analytic Entropy Conjecture},
  pdfauthor={Xiaoyang Chen and Xiang Jiang}}
\numberwithin{equation}{section}
\newtheorem{theorem}{Theorem}[section]
\newtheorem{lemma}[theorem]{Lemma}
\newtheorem{proposition}[theorem]{Proposition}
\newtheorem{corollary}[theorem]{Corollary}
\theoremstyle{definition}
\newtheorem{definition}[theorem]{Definition}

\theoremstyle{remark}

\newcommand{\R}{\mathbb R}
\newcommand{\Z}{\mathbb Z}
\newcommand{\Q}{\mathbb Q}
\newcommand{\T}{\mathbb T}

\newcommand{\cG}{\mathcal G}
\newcommand{\cR}{\mathcal R}
\newcommand{\cB}{\mathcal B}
\newcommand{\cC}{\mathcal C}
\DeclareMathOperator{\Crit}{Crit}
\DeclareMathOperator{\Exp}{Exp}
\DeclareMathOperator{\rank}{rank}
\DeclareMathOperator{\vol}{vol}
\DeclareMathOperator{\diam}{diam}
\DeclareMathOperator{\adj}{adj}
\DeclareMathOperator{\pr}{pr}
\newcommand{\htop}{h_{\mathrm{top}}}
\newcommand{\dd}{\,\mathrm d}
\setlist[enumerate]{leftmargin=*,itemsep=4pt,topsep=5pt}
\setlist[itemize]{leftmargin=*,itemsep=4pt,topsep=5pt}
\title{A Proof of the Analytic Entropy Conjecture}
\author{Xiaoyang Chen \qquad Xiang Jiang\\[5pt]
  \small School of Mathematical Science, Tongji University, Shanghai.}
\date{}

\begin{document}
\maketitle
\begin{abstract}
We prove that the geodesic flow of a closed real-analytic Riemannian
manifold has zero topological entropy if it admits globally defined,
real-analytic Liouville integrals. This proves a longstanding conjecture proposed by Bolsinov and Taimanov. 
\end{abstract}

\section{Introduction}\label{sec:intro}
Let $(M^n,g)$, $n\geq2$, be a closed connected real-analytic Riemannian
manifold. Here closed means compact without boundary, and real analyticity
is required of both the manifold and the metric. Write
$\pi:T^*M\to M$ for the bundle projection and set
\[
 E(q,p)=\tfrac12|p|_g^2,\qquad S=E^{-1}(1/2).
\]
Our conventions are $\lambda=p\,dq$, $\omega=-d\lambda$,
$\iota_{X_F}\omega=dF$, and $\{F,G\}=\omega(X_F,X_G)$.
The flow $\varphi_t$ of $X_E$ on $S$ projects
to unit-speed geodesics.

\begin{definition}[Analytic Liouville integrability]\label{def:integrability}
The geodesic flow is analytically Liouville integrable on $S$ if there
are globally defined, single-valued real-analytic functions
\[
 F_1=E,F_2,\ldots,F_n:W\longrightarrow\R
\]
on one open neighborhood $W$ of the entire $S$, such that
\begin{enumerate}[label=(A\arabic*)]
\item $\{F_i,F_j\}=0$ on $W$ for all $i,j$;
\item the ambient differentials $dF_1,\ldots,dF_n$ are linearly
independent on a dense subset of $S$.
\end{enumerate}
Independence is tested in $T_z^*(T^*M)$, rather than after restricting
the covectors to $T_zS$, on which $dE$ vanishes.
\end{definition}

\begin{theorem}[The analytic entropy conjecture]\label{thm:entropy}
Let $(M,g)$ satisfy Definition~\ref{def:integrability}. Then the
topological entropy of its unit-speed geodesic flow is zero:
\[
 \htop(g)=0.
\]
\end{theorem}

This proves a longstanding conjecture proposed by Bolsinov and Taimanov.
The conjecture is recorded in \cite[\S7, Conjecture 2]{BTsusp}; see
\cite{BJ,Pbook} for background. The analytic surface case was proved
by Paternain \cite{P91}. In higher dimensions he proved zero entropy
under a nondegeneracy hypothesis \cite[Theorem 2.2]{P94}. Taimanov used tame
decompositions to obtain geometric simplicity and topological
restrictions \cite[Theorems 3--4]{Taimanov}. The argument below obtains
uniform geodesic counts without imposing nondegeneracy of the chosen
integrals. Smooth Liouville integrability alone is insufficient, even
for a real-analytic metric \cite{BTpositive}.

For $q,y\in M$ and $T>0$, define
\begin{equation}\label{eq:countdef}
 N_{q,y}(T)=\#\{\xi\in T_qM:0<|\xi|_g\leq T,\ \exp_q(\xi)=y\}.
\end{equation}
Distinct initial velocities, including different traversals of a closed
geodesic, are counted separately. Infinite counts are allowed. If
$z\in S$ and $s>0$, the vector $\xi=s z_g^\sharp$ has norm $s$ and
satisfies $\exp_{\pi z}(\xi)=\pi\varphi_s(z)$; $\sharp_g$ denotes the
metric identification of a covector with a tangent vector. Thus the
variable $s$ below is both the duration and the length of a unit-speed
geodesic segment.

\begin{theorem}[Uniform polynomial geodesic counting]\label{thm:count}
Under Definition~\ref{def:integrability}, there are a constant $C<\infty$
and a Borel set $\cG\subset M\times M$ of full product Riemannian measure
such that
\begin{equation}\label{eq:countbound}
 N_{q,y}(T)\leq C(1+T)^n
 \qquad ((q,y)\in\cG,\ T\geq1).
\end{equation}
Both $C$ and $\cG$ are independent of $T$, and $C$ is independent of
the endpoints.
\end{theorem}

Ma\~n\'e's formula turns this estimate into Theorem~\ref{thm:entropy}.
The same estimate also gives a topological consequence.

\begin{corollary}[Rational ellipticity]\label{cor:elliptic}
Under Definition~\ref{def:integrability}, if $M$ is simply connected,
then, for every integer $N\geq0$,
\begin{equation}\label{eq:betabound}
 \beta_M(N):=\sum_{j=0}^N\dim_{\Q}H_j(\Omega M;\Q)
 \leq C'(1+N)^n.
\end{equation}
Here $\Omega M$ is the based-loop space. Consequently $M$ is rationally
elliptic, meaning that $\sum_{j\geq2}\dim_{\Q}(\pi_j(M)\otimes\Q)<\infty$.
\end{corollary}

Simple connectivity is needed only for this corollary. The exponent $n$
in Theorem~\ref{thm:count} is sharp for the full class of manifolds, as
the flat-torus calculation in Section~\ref{sec:motivation-equation}
shows. 

\section{Proof strategy through the flat-torus model}\label{sec:outline}

\subsection{Why count connecting geodesics}
Ma\~n\'e's formula expresses the topological entropy of the geodesic
flow in terms of the average number of connecting geodesics:
\begin{equation}\label{eq:mane-outline}
 \htop(g)=\lim_{T\to\infty}\frac1T
 \log\int_{M\times M}N_{q,y}(T)\dd\vol_g(q)\dd\vol_g(y).
\end{equation}
A polynomial upper bound with a single constant for almost all endpoint
pairs makes this limit zero. The flat torus suggests both the exponent
and the way to obtain such a uniform estimate.

\subsection{Locate the geodesics that must be counted}
Write $f=(F_2,\ldots,F_n)|_S$. The commuting Hamiltonian fields
identify each connected regular fiber of $f$ with an invariant
Liouville torus. We will organize these tori into finitely many
product families and count geodesics within them.

This reduction holds for one Borel full-measure set
$\cG\subset M\times M$, independent of the length cutoff.
For $(q,y)\in\cG$, all connecting geodesics lie in the retained
families and are nonconjugate at their endpoints. Nonconjugacy
makes their initial vectors isolated in the fiber of
$\Exp(q,\xi)=(q,\exp_q\xi)$; compactness of bounded velocity
disks then makes $N_{q,y}(T)$ finite for each finite $T$.
Section~\ref{sec:regular} constructs $\cG$ by dimension estimates
and Sard's theorem. The task that remains is to make this finiteness
uniform in both endpoints and in $T$.

\subsection{The flat-torus model: geometry and counting}
\label{sec:motivation-equation}
The following calculation isolates the ingredients used in the
general proof: integer winding labels, invariant torus families,
an endpoint equation, and two uniform counting bounds.

\subsubsection{Start with a flat torus}\label{sec:flat-lattice}
On $\R^n/\Z^n$ with its standard flat metric, choose lifts
$\widetilde q,\widetilde y\in\R^n$. A connecting geodesic has
displacement $\widetilde y-\widetilde q+k$ for a unique
$k\in\Z^n$, so
\begin{equation}\label{eq:flatcount}
 N_{q,y}(T)=\#\{k\in\Z^n:0<|\widetilde y-\widetilde q+k|\leq T\}.
\end{equation}
Comparison with Euclidean balls gives growth of order $T^n$ as
$T\to\infty$. This explains the integer label and shows that the
exponent in Theorem~\ref{thm:count} is sharp for the full class of
manifolds considered there.

\subsubsection{A period matrix can diverge even in the flat example}
\label{sec:flat-divergence}
Specialize to $M=\R^2/\Z^2$ and write $[x]$ for reduction
modulo $\Z^2$. In $T^*M=M\times\R^2$, choose
\[
 F_1=E=\tfrac12(p_1^2+p_2^2),\qquad F_2=p_1.
\]
On the regular part of the unit energy surface, write
\[
 p=p_\varepsilon(b)=(b,p_2),\qquad
 p_2=\varepsilon\sqrt{1-b^2},\qquad
 b\in(-1,1),\quad \varepsilon\in\{+1,-1\}.
\]
Fix $\varepsilon$. Flowing for times $a_1,a_2$ along
$X_1=X_E$ and $X_2=X_{F_2}$ changes the lifted position by
$V(b)a$, where $a=(a_1,a_2)$ and
\[
 V(b)=\begin{pmatrix}b&1\\p_2&0\end{pmatrix},\qquad
 P(b)=V(b)^{-1}=\begin{pmatrix}0&1/p_2\\1&-b/p_2\end{pmatrix}.
\]
Thus $a=P(b)u$ converts angular displacement into flow times.
An integer angular loop $k\in\Z^2$ corresponds to the return
time vector $P(b)k$; the columns of $P(b)$ form a basis of the
full period lattice $P(b)\Z^2$.

For representatives $u,v\in[0,1)^2$, geodesic motion for time
$s$ reaches $v+k$ exactly when $u+sp=v+k$. Equivalently,
with $e_1=(1,0)$ denoting the first flow-time direction,
\begin{equation}\label{eq:flatequation}
 P(b)(v-u+k)=s e_1
 \quad\Longleftrightarrow\quad v-u+k=sp.
\end{equation}
As $b\to\pm1$, the matrix $P(b)$ diverges, although the geodesic
still has unit speed. A uniform bound on this matrix is therefore
unavailable even in this elementary model. We will instead bound
the angular winding vector directly.

\subsubsection{Torus families and the common parameter at the endpoints}
\label{sec:flat-families}
Here $S=M\times S^1$ and $f(q,p)=p_1$. For each
$b\in(-1,1)$, the regular fiber has two components
$L_{\varepsilon,b}=M\times\{p_\varepsilon(b)\}$. They form the
two product families
\[
 \tau_\varepsilon:(-1,1)\times\T^2\longrightarrow U_\varepsilon,
 \qquad \tau_\varepsilon(b,\theta)=(\theta,p_\varepsilon(b)),
 \qquad U_\varepsilon=\{(q,p)\in S:\varepsilon p_2>0\}.
\]
These are the flat instance of Proposition~\ref{prop:families}:
the family index distinguishes the two components with the same
integral value $b$. Such a torus lies in phase space, even though
in this example its projection to $M$ is a diffeomorphism.

Set $Z_\varepsilon(b,u)=\tau_\varepsilon(b,[u])$.
Since $\varphi_s([x],p)=([x+sp],p)$ preserves momentum, the
two endpoints of one lifted geodesic are
$Z_\varepsilon(b,u)$ and $Z_\varepsilon(b,v)$ with the same
$b=p_1$ and sign $\varepsilon$, while their base positions
$q=[u]$ and $y=[v]$ may differ. The value of $b$ can change
between different geodesics joining the same $q,y$.

The omitted set is $\Sigma=M\times\{(1,0),(-1,0)\}$ and
consists of horizontal trajectories. There are no conjugate
endpoints for the flat metric. Thus the good endpoint set used in
Lemma~\ref{lem:goodpairs} is simply
\[
 \cG=\{(q,y):q_2\neq y_2\text{ in }\R/\Z\}.
\]
It has full measure, and every connecting geodesic for these pairs
belongs to one of the two regular families, for every length cutoff.

\subsubsection{The coframe interpretation of the endpoint equation}
\label{sec:flat-coframe}
To generalize \eqref{eq:flatequation}, interpret $P(b)u$ as an
integral. On $L_{\varepsilon,b}$ the forms dual to $X_1,X_2$ are
\[
 \alpha_1=\frac{dq_2}{p_2},\qquad
 \alpha_2=dq_1-\frac{b}{p_2}\,dq_2,
 \qquad \alpha_i(X_j)=\delta_{ij}.
\]
Here $dq_j$ is the invariant one-form on the $j$th circle.
If a path has velocity $c_1X_1+c_2X_2$, its integral of
$\alpha_i$ records the accumulated $i$th flow time.
Along $\gamma_{b,u}(r)=([ru],p_\varepsilon(b))$, $0\leq r\leq1$,
this gives
\[
 A(b,u)=\left(\int_{\gamma_{b,u}}\alpha_1,
                   \int_{\gamma_{b,u}}\alpha_2\right)=P(b)u.
\]
In particular, $A(b,e_j)$ is the $j$th period column, since the
path to $e_j$ goes once around the $j$th angular circle.
Equation~\eqref{eq:flatequation} can therefore be written as
\[
 A(b,v)-A(b,u)+P(b)k=s e_1.
\]
This is Corollary~\ref{cor:endpoint} in the flat case. The point
of this reformulation is that integrating the dual coframe still
measures flow time when the angular coordinates no longer
linearize the flow.

\subsubsection{The two bounds behind lattice counting}
\label{sec:flat-solutions}\label{sec:flat-winding}
Fix $(q,y)\in\cG$ and their representatives $u,v\in[0,1)^2$.
First, every segment of length $s\leq T$ satisfies
\[
 k=sp+u-v,\qquad \|k\|_\infty\leq s+1\leq T+1,
 \qquad \|k\|_\infty=\max(|k_1|,|k_2|).
\]
Hence only $O(T^2)$ winding labels are possible, uniformly even
as $p_2\to0$ and $P(b)$ diverges.

Second, for a fixed label $k$, put $d_k=v-u+k$.
The unit-speed condition determines the only possible solution:
\[
 s=|d_k|,\qquad p=\frac{d_k}{|d_k|},\qquad
 b=\frac{(d_k)_1}{|d_k|}.
\]
The definition of $\cG$ ensures $(d_k)_2\neq0$, so $d_k\neq0$
and the sign of $(d_k)_2$ selects exactly one family. This solution
is counted if and only if $|d_k|\leq T$. Thus each label contributes
at most one geodesic across the two families. The formula for $b$
also shows directly why different labels can select different tori
for the same base endpoints.

The general proof retains this separation: bound the number of
possible integer labels, then bound the number of solutions for
each label. Neither step will require a bounded period matrix.

\subsection{From the flat model to the general case}
\label{sec:motivation-bounds}
Three features of the flat calculation require replacements:
the explicit torus products, the linear change to flow times,
and the explicit solution for each winding label.

\paragraph{Finite families replace the two explicit products.}
For $f=(F_2,\ldots,F_n)|_S$, Hardt's theorem organizes the
regular tori into finitely many products
$\tau_\ell:B_\ell\times\T^n\to U_\ell$, after omitting an
invariant set of dimension at most $2n-2$. Here $b\in B_\ell$
is the integral value, and $\ell$ also specifies a fiber component.
This extends the two signs in Section~\ref{sec:flat-families}.
Section~\ref{sec:regular} uses the dimension bound to exclude
only a null set of endpoint pairs and obtain the set $\cG$
described above.

\paragraph{Coframe integrals replace matrix multiplication.}
Fix one family $\tau:B\times\T^n\to U$ and suppress its index.
Hardt's map is a homeomorphism, so its angular coordinates need
not make the flow linear. In contrast, the times $a_i$ of the
commuting fields $X_i=X_{F_i}$ give a universal cover on which
geodesic motion is $a\mapsto a+s e_1$. The dual coframe converts
angular paths into these flow-time coordinates, just as in
Section~\ref{sec:flat-coframe}.

Write $\rho(u)=(e^{2\pi i u_1},\ldots,e^{2\pi i u_n})$ and
$\gamma_{b,u}(r)=\tau(b,\rho(ru))$, $0\leq r\leq1$.
For $u\in Q=[0,1]^n$, define
\begin{equation}\label{eq:outline-A}
 A_i(b,u)=\int_{\gamma_{b,u}}\alpha_i,\qquad
 A=(A_1,\ldots,A_n),\qquad
 P(b)=[\,A(b,e_1)\ \cdots\ A(b,e_n)\,].
\end{equation}
Section~\ref{sec:endpoint} justifies these integrals even though
the coordinate paths need not be smooth. It identifies the period
columns and proves the exact endpoint equation
\begin{equation}\label{eq:outline-endpoint}
 A(b,v)-A(b,u)+P(b)k=s e_1,\qquad k\in\Z^n.
\end{equation}
Only values on one angular cube are needed; long durations enter
through $s e_1$. The common $b$ at the two endpoints now follows
from conservation of the first integrals, which replaces constant
momentum in the flat model.

\paragraph{Uniform finiteness replaces the explicit estimates.}
The two bounds in Section~\ref{sec:flat-solutions} have distinct
generalizations.
\begin{enumerate}
\item \textbf{Bound the winding labels.}
Unit speed on $M$ does not directly control displacement in the
angular coordinates of $\tau$. Instead, follow one angular circle
coordinate during a segment of duration at most one and record
the times when it equals $1\in S^1$. Analyticity gives one
globally subanalytic family of such time sets; Lemma~\ref{lem:uniform}
bounds their numbers of connected components. A real lift can
attain only a uniformly bounded number of distinct integers, since
each needs a different component. Its displacement is therefore
bounded. Concatenating short segments gives
$\|k\|_\infty\leq C_0(1+s)$ and hence $O(T^n)$ labels.
Section~\ref{sec:winding} proves uniformity on each entire family,
including near singular fibers.

\item \textbf{Bound the solutions for one label.}
The projection of a Liouville torus to $M$ need not be a
diffeomorphism, so the base endpoints no longer determine their
angular representatives. For fixed $q,y,T,k$, solve for all of
$(b,u,v,s)$ with $u,v\in[0,1)^n$, imposing
\[
 \pi\tau(b,\rho(u))=q,\qquad \pi\tau(b,\rho(v))=y,
 \qquad 0<s\leq T,
\]
together with \eqref{eq:outline-endpoint}. For $(q,y)\in\cG$,
these solutions are finite because they correspond injectively to
actual connecting geodesics. To bound them uniformly, place every
integer label in one family of real equations by replacing $k$
temporarily with $\kappa\in\R^n$. Theorem~\ref{thm:integration}(i)
describes $A,P$ by globally subanalytic functions and logarithms;
part (ii) bounds all finite solution sets of this family by one
constant. Section~\ref{sec:multiplicity} verifies this application.
\end{enumerate}
Summing over the finitely many torus families gives
\[
 \underbrace{O(T^n)}_{\text{winding labels}}
 \ \times\ \underbrace{O(1)}_{\text{solutions per label and family}}
 \ =\ O(T^n).
\]

\subsection{The precise role of real analyticity}\label{sec:analyticity}
The replacements above use analyticity at three specific points.
\begin{enumerate}
\item \textbf{Organizing the regular tori.}
The analytic metric and first integrals give globally subanalytic
data on the compact unit energy surface. Finite decomposition and
Hardt's theorem supply finitely many torus products and control
the dimension of the omitted set; see Section~\ref{sec:regular}.
\item \textbf{Controlling short orbit segments.}
The analytic metric gives an analytic flow. On each fixed bounded
time interval it is globally subanalytic (Lemma~\ref{lem:flow}).
For durations at most one, this lets Lemma~\ref{lem:uniform}
bound the circle-coordinate time sets used in
Section~\ref{sec:winding}. No description of the entire
unrestricted-time flow is required.
\item \textbf{Controlling the endpoint equations.}
The analytic commuting fields give a dual coframe with globally
subanalytic coefficients on the regular region. Integration along
the coordinate paths yields the functions $A,P$ covered by
Theorem~\ref{thm:integration}. Sections~\ref{sec:endpoint} and
\ref{sec:multiplicity} verify absolute integrability and the
hypotheses of uniform finiteness, respectively.
\end{enumerate}
The first and third steps require analytic first integrals as well
as an analytic metric. Integration can introduce logarithms, so
one cannot assume that $A,P$ remain globally subanalytic.
Section~\ref{sec:ominimal} states the tools needed for precisely
these three steps; Section~\ref{sec:countproof} applies them.

\section{Real-analytic tools for the counting argument}\label{sec:ominimal}
The analytic hypotheses enter the proof through three concrete facts.
Analytic data on compact domains admit finite geometric decompositions;
analytic flows retain this property on a fixed bounded time interval;
and integrating the resulting functions introduces only a controlled
class of logarithmic terms. We state the needed results in that setting.
The order is: identify the analytic data, organize their fibers, and
then control the equations obtained by integration.

\subsection{Compact analytic data and flows on bounded time intervals}
The natural geometric language for these facts is \emph{global
subanalyticity}. Roughly, globally subanalytic sets are the sets obtained
from real-analytic data on compact domains by finite geometric
operations, including projection. A map is globally subanalytic when
its graph is such a set.\footnote{For a precise geometric definition,
a semianalytic set is locally a finite combination of real-analytic
equalities and inequalities; a subanalytic set is locally the projection
of a relatively compact semianalytic set. A set $A\subset\R^m$ is
globally subanalytic if its image under the coordinatewise map
$x_j\mapsto x_j/\sqrt{1+x_j^2}$ is subanalytic in the ambient $\R^m$,
including near the boundary of $[-1,1]^m$.}
The global condition controls behavior near the boundary of a domain
as well as at infinity. It allows $1/x$ on $x>0$, but excludes
$\sin(1/x)$ on $0<x<1$, whose zeros accumulate at the boundary.
Thus an unbounded coefficient is allowed; infinitely many oscillations
are the obstruction relevant here.

We use the following standard closure properties \cite{vdD}.
Finite unions, intersections, complements, products, and coordinate
projections preserve global subanalyticity. Consequently so do
composition, inversion of a bijection, and algebraic operations on
functions, with division restricted to a nonzero denominator.
Partial derivatives, on the sets where they exist, are globally
subanalytic as well. These statements apply to the whole specified
domain, even when it is not compact.

Fix real-analytic Euclidean embeddings of $M$ and $S$, using
\cite[Theorem 3]{Grauert}. An analytic map on a compact analytic
manifold has a compact subanalytic graph and is therefore globally
subanalytic. This applies to $f=(F_2,\ldots,F_n)|_S$ and $\pi|_S$.
These assertions concern restrictions to the compact manifold $S$;
analyticity on the possibly noncompact neighborhood $W$ alone would
not imply global subanalyticity on all of $W$.
For an embedding $\iota:S\hookrightarrow\R^a$, an analytic vector
field is represented by the analytic map
$z\mapsto d\iota_z(X(z))$ on compact $S$. In particular all the
fields $X_i$ have globally subanalytic Euclidean components.
Their Gram matrix can subsequently be inverted on its entire regular
locus, although the inverse may be unbounded there.

\begin{lemma}[Analytic flows on fixed bounded time intervals]\label{lem:flow}
Let $V$ be an analytic vector field on a compact analytic manifold $P$
without boundary, with flow $\psi_t$. For each fixed $R>0$, the map
\[
 [-R,R]\times P\longrightarrow P,\qquad (t,z)\longmapsto\psi_t(z)
\]
is globally subanalytic after an analytic Euclidean embedding of $P$.
\end{lemma}
\begin{proof}
Compactness makes $V$ complete. Analytic dependence for ordinary
differential equations makes the flow map analytic on $\R\times P$.
Its restriction to $[-R,R]\times P$ has a compact subanalytic graph,
which proves the assertion.
\end{proof}

The time interval must be fixed before this lemma is applied.
For example, $t\mapsto e^{2\pi it}$ has the asserted property on
every bounded interval, whereas its unrestricted graph has infinitely
many intersections with $\R\times\{1\}$. In Section~\ref{sec:regular}
we apply the lemma separately with $R=1,2,\ldots$ to exclude a
countable union of null endpoint sets. In Section~\ref{sec:winding}
only $R=1$ is used; longer orbit segments are then concatenated.

\subsection{Finite decompositions of analytic families}
\paragraph{Dimension and sets of measure zero.}
Globally subanalytic sets admit finite partitions into smooth pieces.
We also use the one-variable monotonicity theorem: a globally subanalytic
real-valued function on an interval is $C^1$ and either constant or
strictly monotone on each open interval of a finite partition
\cite{vdD}. Endpoint derivatives need not be bounded. This statement
will justify the path integrals in Lemma~\ref{lem:AC}.
The dimension of a set is the largest dimension of its smooth pieces. For a
globally subanalytic map $h:X\to Y$ and a globally subanalytic
subset $A\subset Y$,
\begin{equation}\label{eq:dimension}
 \dim h(X)\leq\dim X,\qquad
 \dim h^{-1}(A)\leq\dim A+r
\end{equation}
if the fibers over $A$ have dimension at most $r$.
A globally subanalytic subset of dimension less than $d$ in a smooth
$d$-manifold has zero $d$-dimensional measure, by its finite smooth
partition. Section~\ref{sec:regular} applies these facts first to the
omitted torus fibers and then to the endpoint pairs reached from them.

\paragraph{Hardt triviality: finitely many product families.}
If $h:X\to Y$ is continuous and globally subanalytic, its image has
a finite partition into globally subanalytic pieces $B_\nu$ such that
\begin{equation}\label{eq:hardt}
 H_\nu:B_\nu\times K_\nu\longrightarrow h^{-1}(B_\nu),
 \qquad h(H_\nu(b,x))=b,
\end{equation}
where $H_\nu$ is a globally subanalytic homeomorphism and $K_\nu$
is one fixed model fiber. This is Hardt's triviality theorem;
see \cite[Theorem 4.2]{Basu}. The formulation for a general map
follows from the projection formulation by applying it to the graph
of $h$; properness is not required. The pieces can be refined into smooth
cells. A cell here is simply a piece of the finite decomposition
homeomorphic to an open box of its dimension. A full-dimensional
cell in $\R^d$ is open in $\R^d$.

The assertion means that, over one parameter piece, an entire fiber
can be identified continuously with the same model $K_\nu$.
For example, for $h(x_1,x_2)=x_1^2+x_2^2$, the fibers over $b>0$
are circles, and
\[
 H:(0,\infty)\times S^1\longrightarrow\R^2\setminus\{0\},
 \qquad H(b,\theta)=\sqrt b\,\theta
\]
is such a product description. The fiber at $b=0$ is treated separately.
For our map $f:S\to\R^{n-1}$, the connected regular fibers are
tori. Hardt's theorem organizes them into finitely many products
$B_\ell\times\T^n$, as shown in Section~\ref{sec:regular}.
The resulting coordinates are continuous; smoothness and bounds on
their derivatives near a singular value are not part of the theorem.

\begin{lemma}[Uniform bounds for subanalytic families]\label{lem:uniform}
Let $\mathcal P\subset\R^d$ and
$\mathcal Z\subset\mathcal P\times\R^a$ be globally subanalytic.
For $p\in\mathcal P$, write
$\mathcal Z_p=\{x:(p,x)\in\mathcal Z\}$.
There is one integer $K<\infty$ such that every $\mathcal Z_p$ has
at most $K$ connected components. In particular, every finite
$\mathcal Z_p$ has at most $K$ points.
\end{lemma}
\begin{proof}
Apply Hardt's theorem to the parameter projection
$\mathcal Z\to\mathcal P$. Over each of finitely many parameter
pieces all fibers are homeomorphic to one model. Each model has
finitely many connected components by the finite decomposition
property. Take the maximum of these finitely many numbers; empty
fibers contribute zero. The components of a finite set are its points.
\end{proof}

The hypothesis concerns the whole family in the combined variables
$(p,x)$. Separate analytic descriptions for each fixed $p$ would not
give a common bound. Neither compactness of the parameter domain nor
finiteness of all fibers is required. Section~\ref{sec:winding} uses
this lemma for the times when a short angular orbit meets a fixed
point of a circle. After integration, a version allowing logarithms
will be needed for the endpoint equations.

\subsection{Integration, logarithms, and uniform solution counts}
An integral of globally subanalytic functions need not remain
globally subanalytic. Already
\[
 \int_1^x\frac{dr}{r}=\log x\qquad(x\geq1)
\]
leaves that class. The appropriate enlargement for this proof consists
of functions of the form
\begin{equation}\label{eq:constructible}
 C(x)=\sum_{\nu=1}^N g_\nu(x)
              \prod_{j=1}^{r_\nu}\log h_{\nu j}(x),
 \qquad h_{\nu j}(x)>0,
\end{equation}
where all $g_\nu$ and $h_{\nu j}$ are globally subanalytic on the
specified domain. These are called \emph{constructible functions}
in the integration literature \cite{CM}. An empty product is $1$,
so globally subanalytic functions are included. Sums, products, and
composition with a globally subanalytic map preserve this class.

\begin{theorem}[Integration and finiteness for the resulting equations]
\label{thm:integration}
\leavevmode
\begin{enumerate}[label=(\roman*)]
\item Let $X$ be globally subanalytic and let
$a:X\times\R\to\R$ be globally subanalytic. If
$a(x,\cdot)$ is absolutely Lebesgue integrable for every $x\in X$,
then
\[
 I(x)=\int_{\R}a(x,r)\dd r
\]
is constructible on $X$.
\item Let $D\subset\R^d\times\R^a$ be globally subanalytic,
with variables $(p,x)$. Suppose $\mathcal Z\subset D$ is described
by finitely many conditions $C(p,x)=0$ or $C(p,x)>0$, where each
$C$ is constructible on $D$, with finite unions, intersections, or
complements relative to $D$ allowed. There is one integer $K<\infty$ such that
every solution set $\mathcal Z_p$ has at most $K$ connected
components. Every finite solution set therefore has at most $K$ points.
\end{enumerate}
\end{theorem}
\begin{proof}[Sources and justification]
Part (i) is the required special case of Cluckers--Miller's integration
theorem \cite[Theorem 1.3]{CM}. It also follows from the volume theorem
of Comte--Lion--Rolin \cite{CLR}: the positive and negative parts
$a^\pm$ give globally subanalytic subgraph families
\[
 \{(r,t):0<t<a^\pm(x,r)\}.
\]
Their finite areas are constructible functions of $x$. Their difference
is the stated integral by Tonelli's theorem.

Part (ii) uses the o-minimality theorem for restricted real-analytic
functions and exponentiation \cite[Corollary 5.13]{DMM}, together
with cylindrical cell decomposition \cite{vdD}. Here restricted analytic
functions are analytic functions used on compact boxes to which they
extend analytically; logarithms are included as inverses of the
exponential. The consequence needed here is that the combined solution
set has a finite cylindrical cell decomposition with the parameter coordinates
first. At fixed $p$, each cell contributes either nothing or one
connected cell. The total number of cells therefore bounds the number
of connected components of every solution set. If the solution set
is finite, each nonempty cell fiber is a point. Every function in
\eqref{eq:constructible} belongs to the analytic and logarithmic
setting of that theorem, which verifies its applicability.
\end{proof}

In part (ii), the bound applies to the number of connected components
of every fiber, including infinite fibers. Applying it to a finite
fiber gives a point count because each of its components is a point.
Thus no restriction to a separately described set of finite fibers
is required.

The two parts have different roles and different sources.
Part (i) controls the functions produced by integration; part (ii)
controls the number of solutions of equations involving those functions.
In particular, part (ii) is not an assertion that those solution sets
are globally subanalytic.

For part (i), integrability is checked separately at each parameter;
no common $L^1$ bound is required. In Section~\ref{sec:endpoint},
the parameter is $x=(b,u)$ and the variable $r\in[0,1]$ traces
a path on one compact regular torus. Extending the integrand by
zero outside $[0,1]$ gives the theorem's formulation on $\R$.
Its output is $A(b,u)$, and evaluation at $u=e_j$ gives the columns
of $P(b)$. Both have constructible entries even if they diverge
near singular fibers.

In Section~\ref{sec:multiplicity}, part (ii) applies to the equations
for $(b,u,v,s)$ with real parameters $(q,y,T,\kappa)$.
After obtaining one bound for every finite solution set, we use it
at the good endpoints and integer values $\kappa=k$.
No description of the selected parameters by analytic equations is
required. This order is essential: the integers themselves do not
belong to the finite geometric setting above.

\subsection{The tools as used in the proof}
\begin{center}
\small
\begin{tabular}{@{}>{\raggedright\arraybackslash}p{0.25\textwidth}
  >{\raggedright\arraybackslash}p{0.40\textwidth}
  >{\raggedright\arraybackslash}p{0.25\textwidth}@{}}
\toprule
Input & Consequence & Use\\
\midrule
Analytic metric and integrals & Compact subanalytic data and finite torus products
 & Hardt's theorem, \S\ref{sec:regular}\\[4pt]
Lemma~\ref{lem:flow} & Subanalytic flow maps on fixed time intervals
 & Null endpoint sets and short orbits, \S\ref{sec:regular}, \S\ref{sec:winding}\\[4pt]
Lemma~\ref{lem:uniform} & Uniformly many components of circle-cut time sets
 & Linear winding bound, \S\ref{sec:winding}\\[4pt]
Theorem~\ref{thm:integration}(i) & Constructible entries of $A$ and $P$
 & Exact endpoint equation, \S\ref{sec:endpoint}\\[4pt]
Theorem~\ref{thm:integration}(ii) & Uniformly bounded finite solution sets
 & Multiplicity for one label, \S\ref{sec:multiplicity}\\
\bottomrule
\end{tabular}
\end{center}
\section{Proof of the polynomial counting theorem}\label{sec:countproof}
We now implement Section~\ref{sec:outline}. First choose finite torus
families and good endpoints. Next express a geodesic by an endpoint
equation whose coefficients are controlled by Theorem~\ref{thm:integration}(i).
Finally bound the number of winding labels using Lemma~\ref{lem:uniform},
and the multiplicity of each label using Theorem~\ref{thm:integration}(ii).

Section~\ref{sec:motivation-equation} develops the flat-torus model
that motivates the invariant families, endpoint equation, and two
counting bounds. Section~\ref{sec:motivation-bounds} explains the
changes needed in general. We now prove those steps, referring back
to the model where it clarifies the construction.

\subsection{Regular tori and a fixed set of good endpoints}\label{sec:regular}
Throughout this section,
\[
 f=(F_2,\ldots,F_n)|_S:S\to\R^{n-1},\qquad
 X_i=X_{F_i}|_S\quad(1\leq i\leq n).
\]
Poisson commutation makes the fields $X_i$ tangent to $S$, preserving
$f$, and mutually commuting. Compactness of $S$ makes them complete.
Since $dE\neq0$ on $S$, a relation among
$dF_2|_{T_zS},\ldots,dF_n|_{T_zS}$ holds precisely when the
corresponding ambient linear combination is a multiple of $dE(z)$.
Thus ambient independence of $dE,dF_2,\ldots,dF_n$ is equivalent
to $\rank df(z)=n-1$. Nondegeneracy of $\omega$ makes it equivalent
to independence of $X_1(z),\ldots,X_n(z)$ as well.

\begin{samepage}
\subsubsection{Finitely many full families of Liouville tori}
\begin{proposition}[Finite torus families]\label{prop:families}
There are finitely many globally subanalytic homeomorphisms
\begin{equation}\label{eq:families}
 \tau_\ell:B_\ell\times\T^n\longrightarrow U_\ell,
 \qquad f(\tau_\ell(b,\theta))=b,\qquad 1\leq\ell\leq L,
\end{equation}
where each $B_\ell\subset\R^{n-1}$ is a cell of dimension $n-1$
consisting of regular values with nonempty fibers. The sets $U_\ell$
are disjoint open subsets of $S$, made up of whole connected fiber
components, and invariant under every $X_i$. Their complement
\begin{equation}\label{eq:sigma}
 \Sigma=S\setminus\bigcup_{\ell=1}^L U_\ell
 \quad\text{is closed and satisfies}\quad \dim\Sigma\leq2n-2.
\end{equation}
The constants later obtained on these products apply on all of
$B_\ell$, without restricting to compact subsets.
\end{proposition}
\end{samepage}
\begin{proof}
\emph{Control the fibers over critical values.}
Let $D=\Crit(f)$ be the critical locus and $\Delta=f(D)$ the set of
critical values. The rank condition is analytic, and the regular locus
is dense, so $D$ is a proper analytic set of dimension at most $2n-2$.
Both $D$ and $\Delta$ are compact and globally subanalytic. By Sard's
theorem $\Delta$ has measure zero. Its subanalytic stratification
therefore has no $(n-1)$-dimensional piece, so $\dim\Delta\leq n-2$.
On $S\setminus D$, the map $f$ is a submersion with $n$-dimensional
fibers. The fiber dimension estimate \eqref{eq:dimension} gives
\[
 \dim\bigl(f^{-1}(\Delta)\setminus D\bigr)
 \leq n+\dim\Delta\leq2n-2.
\]
Together with the bound for $D$, this yields
\begin{equation}\label{eq:badfiberdim}
 \dim f^{-1}(\Delta)\leq2n-2.
\end{equation}
This step removes entire fibers over critical values, not only their
critical points; a regular point can lie over a critical value.

\emph{Identify each regular connected component as a torus.}
For a regular value $b$, a connected component $L_b$ of $f^{-1}(b)$
is a compact $n$-manifold. The fields $X_1,\ldots,X_n$ are tangent
to it and form a basis at every point. Their commuting complete flows
define an $\R^n$-action. Every orbit is open, so connectedness makes
the action transitive. The stabilizer $\Lambda_b$ of a point is closed
and discrete, the latter because the fields are independent. The orbit map induces a
bijective local analytic diffeomorphism $\R^n/\Lambda_b\to L_b$,
hence an analytic diffeomorphism. Its quotient is therefore compact,
which forces $\Lambda_b$ to be a lattice of rank $n$. This identifies
$L_b$ analytically with an $n$-torus.

\emph{Use Hardt's theorem to choose subanalytic product families.}
Apply \eqref{eq:hardt} to $f$ over $f(S)\setminus\Delta$, refining
its finite partition into cells. On one cell $B$, write its product
homeomorphism as $H:B\times K\to f^{-1}(B)$ and set
$H_b(x)=H(b,x)$. Fix $b_0\in B$. For each connected component
$K_0$ of $K$, the manifold $L_0=H_{b_0}(K_0)$ is a regular torus.
Choose a basis of its period lattice and the resulting analytic
identification $j_0:\T^n\to L_0$. Its graph is compact and globally
subanalytic. Then
\[
 \tau(b,\theta)=H_b\bigl(H_{b_0}^{-1}(j_0(\theta))\bigr)
\]
is a globally subanalytic homeomorphism from $B\times\T^n$ onto
the corresponding family of connected components. There are finitely
many components and cells, so this gives finitely many products.
The analytic identification is chosen on one fixed fiber only; Hardt's
homeomorphism transports it topologically to the others.

\emph{Retain full-dimensional cells and estimate what was omitted.}
Keep the cells of dimension $n-1$. Each omitted regular-value cell
has dimension at most $n-2$, so its inverse image has dimension at
most $n+(n-2)=2n-2$. Along with \eqref{eq:badfiberdim}, this proves
the dimension estimate in \eqref{eq:sigma}.
A full-dimensional cell is open in $\R^{n-1}$. Each component of
the compact reference fiber is open and closed in that fiber because
there are only finitely many components. The product homeomorphism
therefore makes each retained $U_\ell$ open in $f^{-1}(B_\ell)$,
which is itself open in $S$.
Thus $\Sigma$ is closed. Every $X_i$ preserves
$f$, and its continuous orbit stays in the same connected component
of a fiber. Hence all the sets $U_\ell$ and $\Sigma$ are invariant.
\end{proof}

For the explicit two families on the flat two-torus, including their
component indices and exceptional set, see Section~\ref{sec:flat-families}.

\subsubsection{One endpoint set for every length}
\begin{samepage}
\begin{lemma}[Good endpoint pairs]\label{lem:goodpairs}
There is a Borel full-measure set $\cG\subset M\times M$, independent
of any length bound, such that for every $(q,y)\in\cG$:
\begin{enumerate}[label=(\roman*)]
\item all geodesics from $q$ to $y$ are nonconjugate at their endpoints;
\item $N_{q,y}(T)<\infty$ for every finite $T$;
\item every connecting geodesic, after unit-speed normalization,
starts in one of the $U_\ell$.
\end{enumerate}
\end{lemma}
\end{samepage}
\begin{proof}
For each integer $m\geq1$, consider
\begin{equation}\label{eq:exceptionalmap}
 \Psi_m:\Sigma\times[0,m]\longrightarrow M\times M,
 \qquad \Psi_m(z,s)=(\pi z,\pi\varphi_s(z)).
\end{equation}
Lemma~\ref{lem:flow} makes this map globally subanalytic for this
fixed $m$. Its domain has dimension at most $2n-1$, so its image
has measure zero in $M\times M$. Since $\Sigma$ is compact, the
image is compact. Consequently
\[
 \cB=\bigcup_{m=1}^{\infty}\Psi_m(\Sigma\times[0,m])
\]
is Borel and null. Outside $\cB$, no finite positive-time trajectory
from $\Sigma$ connects the two endpoints. This is the first use of
the fixed-time flow lemma; a countable union suffices here.

Next consider the global exponential endpoint map
\begin{equation}\label{eq:Exp}
 \Exp:TM\longrightarrow M\times M,\qquad
 \Exp(q,\xi)=(q,\exp_q\xi).
\end{equation}
In local bundle coordinates its derivative has block form
\[
 d\Exp_{(q,\xi)}=
 \begin{pmatrix}I&0\\ *&D_\xi\exp_q\end{pmatrix}.
\]
Thus $(q,\xi)$ is regular precisely when the vertical derivative
$D_\xi\exp_q$ is invertible. For $\xi\neq0$, this is equivalent
to the two endpoints being nonconjugate along the corresponding
geodesic \cite[Theorem 18.1]{Milnor}; it does not exclude conjugate
points in the interior of the segment. At $\xi=0$ the derivative
is the identity.
The critical values form a null set $\cC$ by Sard's theorem. They
are Borel: the critical locus is closed, and its intersections with
integer-radius compact disk bundles have compact images whose union
is $\cC$. Define
\begin{equation}\label{eq:goodset}
 \cG=(M\times M)\setminus
 \bigl(\cB\cup\cC\cup\{(q,q):q\in M\}\bigr).
\end{equation}
This is the required Borel set of full measure.

For $(q,y)\in\cG$, every point of $\Exp^{-1}(q,y)$ is regular.
The inverse function theorem makes this fiber discrete, since the
source and target of $\Exp$ have the same dimension. It is also
closed. Its intersection with a compact velocity disk is compact and
discrete, hence finite. Removing the diagonal excludes the zero
velocity; indeed every connecting length is at least $d_g(q,y)>0$.
The exclusion of $\cB$ places every normalized initial covector in
one of the retained families.
\end{proof}

The unit round sphere explains the need for the exceptional set. All
meridians from a point to its antipode are geodesics of length $\pi$,
so there are infinitely many such segments. Their endpoints are
conjugate, and the pair is removed in \eqref{eq:goodset}. The proof
controls the endpoint average; it makes no assertion that discarding
an invariant null set preserves topological entropy.

In the flat model, this construction is the condition $q_2\neq y_2$
from Section~\ref{sec:flat-families}. The argument above replaces
that coordinate description by dimension estimates and Sard's theorem.

\subsection{The endpoint equation in flow-time coordinates}\label{sec:endpoint}
Fix a single family $\tau:B\times\T^n\to U$, and suppress its index
$\ell$ until Section~\ref{sec:multiplicity}. We construct functions
that express a long geodesic segment using its endpoints, length, and
integer winding vector. To keep their roles explicit, the notation
for this construction is collected here.
\begin{center}
\small
\begin{tabular}{@{}>{\raggedright\arraybackslash}p{0.22\textwidth}
  >{\raggedright\arraybackslash}p{0.70\textwidth}@{}}
\toprule
Symbol & Meaning\\
\midrule
$b\in B$ & Value of $f$; labels the torus $L_b=\tau(\{b\}\times\T^n)$.\\[3pt]
$u\in\R^n$ & Lifted angular coordinate, with quotient $\rho(u)\in\T^n$.\\[3pt]
$a\in\R^n$ & Times of the commuting flows $X_1,\ldots,X_n$.\\[3pt]
$Q$, $Q_0$ & $[0,1]^n$ for integration; $[0,1)^n$ for unique representatives.\\[3pt]
$Z(b,u)$, $z_b$ & $\tau(b,\rho(u))$ and the reference point $Z(b,0)$.\\[3pt]
$A(b,u)$ & Flow-time displacement along the chosen path from $z_b$ to $Z(b,u)$.\\[3pt]
$P(b)$ & Matrix of flow-time displacements of the $n$ basic angular loops.\\[3pt]
$s$, $k$ & Geodesic duration and length; integer angular winding vector.\\
\bottomrule
\end{tabular}
\end{center}

\subsubsection{A coframe that measures the commuting flow times}
Use the Euclidean embedding of $S$ fixed in Section~\ref{sec:ominimal}.
On $S^{\mathrm{reg}}=S\setminus\Crit(f)$, define the Gram matrix
$\Gamma$ and the one-forms $\alpha_i$ by
\begin{equation}\label{eq:coframe}
 \Gamma_{ij}(z)=\langle X_i(z),X_j(z)\rangle,\qquad
 \alpha_i|_z(w)=\sum_{j=1}^n(\Gamma(z)^{-1})_{ij}
                         \langle X_j(z),w\rangle,
 \quad w\in T_zS.
\end{equation}
The brackets here are the ambient Euclidean inner product, not the
metric $g$ on $M$. Independence of the fields makes $\Gamma$ positive
definite. Its entries are globally subanalytic, and
$\Gamma^{-1}=\adj(\Gamma)/\det\Gamma$ is globally subanalytic on
the whole set $\det\Gamma>0$. Thus the forms in \eqref{eq:coframe}
have analytic, globally subanalytic coefficients even though that
domain is not compact.

By multiplication of the two matrices,
$\alpha_i(X_j)=\delta_{ij}$. On a fixed torus $L_b$, the $X_j$
span the tangent space, so these restrictions form its dual coframe.
Commutation gives
\[
 d\alpha_i(X_j,X_k)
 =X_j\bigl(\alpha_i(X_k)\bigr)-X_k\bigl(\alpha_i(X_j)\bigr)
       -\alpha_i([X_j,X_k])=0.
\]
Hence
\begin{equation}\label{eq:closed}
 d(\alpha_i|_{L_b})=0.
\end{equation}
Only closedness along a torus is asserted. These forms are bounded
on each fixed compact $L_b$, but the bound may depend on $b$.

\subsubsection{Why the coordinate paths can be integrated}
The map $\tau$ is only a homeomorphism. The paths obtained from it
need not be smooth, so we record the exact regularity used in the
line integrals and the later lifting argument. A Euclidean path is
absolutely continuous when it is the integral of an $L^1$ derivative
plus its initial value; such paths have finite length and satisfy the
usual chain rule for smooth maps locally with bounded derivative.

\begin{samepage}
\begin{lemma}[Regularity of the paths and their lifts]\label{lem:AC}
\leavevmode
\begin{enumerate}[label=(\roman*)]
\item A continuous globally subanalytic path
$\gamma:[0,1]\to\R^a$ is absolutely continuous and has finite length.
\item Let $p:\widetilde L\to L$ be a smooth covering that is a local
diffeomorphism. If $\gamma:[0,1]\to L$ is absolutely continuous,
every continuous lift $\widetilde\gamma$ is absolutely continuous and
\[
 dp_{\widetilde\gamma(r)}\widetilde\gamma'(r)=\gamma'(r)
 \quad\text{for almost every }r.
\]
Absolute continuity on manifolds is understood in smooth local coordinates.
\end{enumerate}
\end{lemma}
\end{samepage}
\begin{proof}
For (i), the one-variable decomposition theorem for globally
subanalytic functions gives a finite partition on whose open intervals
every coordinate of $\gamma$ is $C^1$ and monotone. Let $h$ be one coordinate on a
partition interval $(a,b)$, with its continuous endpoint values.
The fundamental theorem of calculus on compact subintervals, followed
by passage to the endpoints, gives
\[
 \int_a^b|h'(r)|\dd r=|h(b)-h(a)|<\infty,\qquad
 h(x)-h(a)=\int_a^x h'(r)\dd r.
\]
The fixed sign of $h'$ justifies these limits. Thus $h$ is absolutely
continuous on $[a,b]$. Finite gluing proves absolute continuity of
all coordinates on $[0,1]$. The estimate
$|\gamma'|\leq\sum_j|\gamma_j'|$ also gives finite length.

For (ii), cover the compact lifted path by inverse-branch neighborhoods
for $p$. In local coordinates, choose smaller balls with closures in
larger balls on which each inverse branch is smooth. Its derivative
is bounded on a smaller closed ball, so the inverse branch is
Lipschitz there. Continuity and compactness give a finite subdivision
of the parameter interval such that each path piece lies in one of
these neighborhoods. On each piece the lift is a smooth Lipschitz
inverse branch composed with $\gamma$, and is therefore absolutely
continuous. The ordinary chain rule holds almost everywhere on each
piece. Finite gluing proves the claim. All constants are chosen for
this individual lifted path.
\end{proof}

Identify $\T^n=(S^1)^n$ and let
\[
 \rho:\R^n\to\T^n,\qquad
 \rho(u)=(e^{2\pi i u_1},\ldots,e^{2\pi i u_n}),\qquad
 Q=[0,1]^n,\quad Q_0=[0,1)^n.
\]
The full covering map $\rho$ is used topologically. Its restriction
to $Q$ extends analytically to a neighborhood of that cube and is
therefore globally subanalytic. Consequently
\[
 Z:B\times Q\to U,\qquad Z(b,u)=\tau(b,\rho(u)),
 \qquad z_b=Z(b,0)
\]
is globally subanalytic. For $(b,u)\in B\times Q$, set
\begin{equation}\label{eq:Adefinition}
 \gamma_{b,u}(r)=Z(b,ru)\quad(0\leq r\leq1),\qquad
 A_i(b,u)=\int_0^1
       \alpha_i|_{\gamma_{b,u}(r)}\bigl(\gamma_{b,u}'(r)\bigr)\dd r.
\end{equation}
The derivative is interpreted almost everywhere. The path begins at
$z_b$ and stays in the fixed torus $L_b$.
The closed cube includes the vertices $e_j$, which will be used to
define periods. The half-open cube will be used later only to remove
ambiguity in representatives.

\begin{lemma}[Constructible path integrals]\label{lem:Adefinable}
Every integral in \eqref{eq:Adefinition} converges absolutely.
Each component of $A=(A_1,\ldots,A_n):B\times Q\to\R^n$ is a
constructible function of the combined variables $(b,u)$, in the
sense of \eqref{eq:constructible}.
\end{lemma}
\begin{proof}
\emph{Absolute integrability for every parameter.}
Fix $(b,u)$. The path $\gamma_{b,u}$ is continuous and globally
subanalytic, so Lemma~\ref{lem:AC}(i) makes it absolutely continuous
with finite length. Its derivative is tangent to $L_b$ almost
everywhere, as follows by differentiating the vanishing normal
coordinates in a submanifold chart. Since $L_b$ is compact and
$\alpha_i$ is smooth there,
some finite $C_{i,b}$ satisfies
\begin{equation}\label{eq:L1}
 \int_0^1\left|\alpha_i|_{\gamma_{b,u}(r)}
                       (\gamma_{b,u}'(r))\right|\dd r
 \leq C_{i,b}\int_0^1|\gamma_{b,u}'(r)|\dd r<\infty.
\end{equation}
No common bound in $b$ or $u$ is required by
Theorem~\ref{thm:integration}(i).

\emph{A single globally subanalytic integrand.}
At points with $0<r<1$ where the partial derivative in $r$ exists, put
\begin{equation}\label{eq:integrand}
 \mathfrak a_i(b,u,r)=\sum_{j=1}^n
  (\Gamma(\gamma_{b,u}(r))^{-1})_{ij}
  \left\langle X_j(\gamma_{b,u}(r)),\partial_r\gamma_{b,u}(r)\right\rangle.
\end{equation}
Set $\mathfrak a_i=0$ at every other point, including outside $(0,1)$.
The closure property for partial derivatives from
Section~\ref{sec:ominimal} applies to this jointly globally subanalytic
path family. More explicitly, existence and value of this derivative
are expressed by the usual limit condition on difference quotients,
with real quantifiers. Closure under projection and complement makes
its domain and graph globally subanalytic, jointly in $(b,u,r)$.
Combining this derivative with the coframe coefficients and extending
by zero makes $\mathfrak a_i$ globally subanalytic on
$B\times Q\times\R$.
For each fixed $(b,u)$, the path is piecewise $C^1$ with only finitely
many exceptional times. Assigning zero at those times does not alter
its integral. Globally subanalytic functions are measurable by their finite
decomposition into continuous pieces.

\emph{Apply the integration theorem with the correct parameters.}
In Theorem~\ref{thm:integration}(i), take $X=B\times Q$, the parameter
$x=(b,u)$, and the integration variable $r$. Equation~\eqref{eq:L1}
verifies the required hypothesis for every $x$. Hence
\[
 A_i(b,u)=\int_{\R}\mathfrak a_i(b,u,r)\dd r
\]
is constructible, as claimed.
\end{proof}

This is the only step that invokes integration in
Theorem~\ref{thm:integration}(i). The derivative is taken only along
the path variable $r$; neither smoothness of $\tau$ in all its variables
nor differentiation of the integral with respect to $b$ is needed.

\subsubsection{Compare the two universal covers and identify the periods}
The integral just defined must still be related to actual geodesic
motion. This is a covering-space argument for each fixed $b$.

\begin{lemma}[Developing coordinates and the full period lattice]
\label{lem:developing}
For each $b\in B$, there is a homeomorphism
$h_b:\R^n\to\R^n$, with $h_b(0)=0$, such that
\begin{equation}\label{eq:hA}
 h_b(u)=A(b,u)\qquad(u\in Q).
\end{equation}
The matrix with constructible entries
\begin{equation}\label{eq:Pdefinition}
 P(b)=[\,A(b,e_1)\ \cdots\ A(b,e_n)\,]
\end{equation}
is invertible, its columns form a basis of the full flow-time period
lattice, and
\begin{equation}\label{eq:equivariance}
 h_b(u+k)=h_b(u)+P(b)k
 \qquad(u\in\R^n,\ k\in\Z^n).
\end{equation}
In flow-time coordinates $a=h_b(u)$, the lifted geodesic flow is
$a\mapsto a+s e_1$.
\end{lemma}
\begin{proof}
Write $\psi_i^t$ for the flow of $X_i$. Starting at $z_b$, the
commuting action gives an analytic covering
\begin{equation}\label{eq:Theta}
 \Theta_b:\R^n\longrightarrow L_b,\qquad
 \Theta_b(a)=\psi_1^{a_1}\circ\cdots\circ\psi_n^{a_n}(z_b).
\end{equation}
Its stabilizer $\Lambda_b=\Theta_b^{-1}(z_b)$ is the full lattice
from Proposition~\ref{prop:families}; its deck transformations are
translations $a\mapsto a+\lambda$, $\lambda\in\Lambda_b$.
The map $u\mapsto\tau(b,\rho(u))$ is another universal cover of
$L_b$, now with deck transformations $u\mapsto u+k$ for $k\in\Z^n$.
Uniqueness of universal covers gives a homeomorphism $h_b$ with
\begin{equation}\label{eq:covers}
 \Theta_b\circ h_b=\tau_b\circ\rho,
 \qquad h_b(0)=0,\qquad \tau_b(\theta)=\tau(b,\theta).
\end{equation}
This construction is for a fixed $b$. The map $h_b$ on the unbounded
cover is used only in this covering-space argument. The functions
used later in the counting theorem are its values on the cube $Q$,
identified with the constructible functions $A(b,u)$.

Commutativity of the flows implies
$d\Theta_b|_a(e_j)=X_j(\Theta_b(a))$. The duality of the coframe
therefore gives
\begin{equation}\label{eq:pullback}
 \Theta_b^*\alpha_i=da_i,
\end{equation}
where $a_i$ denotes the $i$th flow-time coordinate of $a\in\R^n$.
For $u\in Q$, the path $r\mapsto h_b(ru)$ is the lift of
$\gamma_{b,u}$ starting at zero. Lemma~\ref{lem:AC}(ii) makes this
lift absolutely continuous. By \eqref{eq:pullback}, the almost-everywhere
chain rule yields
\[
 \frac{d}{dr}\bigl(h_b(ru)\bigr)_i
     =\alpha_i|_{\gamma_{b,u}(r)}(\gamma_{b,u}'(r)).
\]
Integrating proves \eqref{eq:hA}. This calculation also explains why
the coframe integral records flow time despite the limited regularity
of the angular trivialization.

For $k\in\Z^n$, equation~\eqref{eq:covers} gives
$h_b(u+k)-h_b(u)\in\Lambda_b$. This difference is continuous in
$u$ and the lattice is discrete, so it is a constant $\lambda_b(k)$.
Conjugation by $h_b$ identifies the two deck groups; conjugation by
$h_b^{-1}$ is its inverse. Consequently
$\lambda_b:\Z^n\to\Lambda_b$ is an isomorphism onto the entire
stabilizer. Setting $u=0$ and $k=e_j$ gives
\[
 \lambda_b(e_j)=h_b(e_j)=A(b,e_j).
\]
Thus the columns of $P(b)$ form a basis of the full lattice, rather
than merely a sublattice. In particular $P(b)$ is invertible, and
additivity proves \eqref{eq:equivariance}. Its entries are constructible
because they are obtained by evaluating $A$ at fixed vertices of $Q$. Finally,
$X_1=X_E$, so geodesic motion in \eqref{eq:Theta} adds $s e_1$ to
the flow-time vector.
\end{proof}

\subsubsection{The exact endpoint equation and its integer label}
\begin{corollary}[Exact endpoint equation]\label{cor:endpoint}
For $b\in B$, $u,v\in Q$, and $s\in\R$,
\begin{equation}\label{eq:endpoint}
 \varphi_s(Z(b,u))=Z(b,v)
 \quad\Longleftrightarrow\quad
 \exists k\in\Z^n:\ A(b,v)-A(b,u)+P(b)k=s e_1.
\end{equation}
For fixed $b,u,v,s$, this integer vector, when it exists, is unique.
\end{corollary}
\begin{proof}
The initial flow-time coordinate is $h_b(u)$, so its endpoint after
time $s$ is $h_b(u)+s e_1$. It projects to $Z(b,v)$ precisely when
\[
 h_b(u)+s e_1=h_b(v)+P(b)k
 \quad\text{for some }k\in\Z^n.
\]
Substitution of \eqref{eq:hA} gives \eqref{eq:endpoint}. Invertibility
of $P(b)$ proves uniqueness.
\end{proof}

For $s\geq0$, the lifted angular path beginning at $u$ is
\begin{equation}\label{eq:liftedorbit}
 \widetilde\theta(t)=h_b^{-1}(h_b(u)+t e_1),\qquad 0\leq t\leq s.
\end{equation}
By \eqref{eq:equivariance} and \eqref{eq:endpoint}, its final point
is $v+k$. This identifies the integer in the equation with the angular
winding vector used in the next section. The two appearances of
$e_j$ have distinct roles: in $A(b,e_j)$ it is a vertex of the
angular cube, while $e_1$ on the right of \eqref{eq:endpoint} is the
first flow-time direction. The matrix $P(b)$ sends integer angular
winding vectors to flow-time periods. The general change from angular
to flow-time coordinates is $h_b$, which need not be linear.

In Section~\ref{sec:flat-coframe}, the coframe integral gives
$A(b,u)=P(b)u$, and this equation reduces to $v-u+k=sp$.
The proof here establishes the same relation in flow-time coordinates
without assuming that the change from angles is linear.

\subsection{A uniform bound on winding}\label{sec:winding}
Equation~\eqref{eq:endpoint} has an integer unknown $k\in\Z^n$.
We first bound how many such vectors can occur for lengths $s\leq T$.
The bound will use only global subanalyticity of short angular
orbit segments.
It will be valid on every retained torus, including arbitrarily close
to singular fibers.

\subsubsection{A short circle path has uniformly bounded lifted displacement}
The next lemma isolates the topological observation. It bounds the
displacement of a real lift. A lift can be chosen separately for
each path; no regularity with respect to the parameters is needed.

\begin{lemma}[Uniform displacement for a family of circle paths]
\label{lem:circle}
Let $D$ be globally subanalytic and let
$c:D\times[0,1]\to S^1$ be globally subanalytic. Assume that
$r\mapsto c(p,r)$ is continuous for every $p\in D$. There is
$C_*<\infty$ such that every continuous lift $x:[0,1]\to\R$ with
$c(p,r)=e^{2\pi i x(r)}$ satisfies
\[
 \sup_{r,r'\in[0,1]}|x(r)-x(r')|\leq C_*.
\]
\end{lemma}
\begin{proof}
Choose the point $1\in S^1$ as a cut and define
\[
 W_p=\{r\in[0,1]:c(p,r)=1\}.
\]
The total set $\{(p,r):r\in W_p\}$ is globally subanalytic, being
the inverse image of the fixed circle point $1$ under $c$.
Lemma~\ref{lem:uniform} therefore gives one integer $K$
bounding the number of connected components of every $W_p$.
For a lift $x$, we have $W_p=x^{-1}(\Z)$. On each connected
component of $W_p$, continuity forces $x$ to be constant at a single
integer.

If $|x(r)-x(r')|>K+2$, the intermediate value theorem forces $x$
to attain more than $K$ distinct integers on the interval between
$r$ and $r'$. Each such integer needs a different component of
$W_p$, a contradiction. Thus $C_*=K+2$ suffices.
\end{proof}

Counting components includes tangencies to the cut and intervals spent
at the cut; no assumption of transverse crossings is present. The
parameter domain $D$ can be noncompact, so the conclusion remains
uniform near the boundary of a regular torus family. The estimate
controls lifted displacement; no bound for the length or total
variation of the lifted path is asserted or needed.

\subsubsection{Apply the lemma to short orbits, then concatenate}
\begin{lemma}[Linear winding bound]\label{lem:winding}
There is a constant $C_0$ such that every endpoint relation
\eqref{eq:endpoint}, in any retained chart, with $s\geq0$ and
$u,v\in Q$, satisfies
\begin{equation}\label{eq:winding}
 \|k\|_\infty\leq C_0(1+s),
 \qquad \|k\|_\infty=\max_{1\leq j\leq n}|k_j|.
\end{equation}
\end{lemma}
\begin{proof}
Fix a chart $U$. Define its $j$th angular projection by
\[
 \vartheta_j:U\longrightarrow S^1,\qquad
 \vartheta_j=\pr_j\circ\pr_{\T^n}\circ\tau^{-1}.
\]
Here $\pr_{\T^n}$ discards the parameter $b$ from
$B\times\T^n$, and $\pr_j$ selects its $j$th circle factor.
For an initial covector $z\in U$, a duration $\sigma\in[0,1]$,
and a normalized path variable $r\in[0,1]$, set
\begin{equation}\label{eq:circlefamily}
 c_j((z,\sigma),r)=\vartheta_j(\varphi_{\sigma r}(z)).
\end{equation}
These are well-defined because $U$ is invariant. Lemma~\ref{lem:flow}
with $R=1$ makes the short-time flow globally subanalytic.
The inverse $\tau^{-1}$ is globally subanalytic as well, so
\eqref{eq:circlefamily} has that property in all variables on
$D\times[0,1]$, where $D=U\times[0,1]$. For each parameter the path
is continuous. Lemma~\ref{lem:circle}, applied for $j=1,\ldots,n$,
gives a single $D_0$ bounding the displacement in every lifted angular
coordinate during any segment of duration at most one.

Now let $s>0$ and choose one continuous lift of the angular orbit
on $[0,s]$. Divide this interval into at most $\lceil s\rceil$ pieces
of duration at most one. Every new initial point remains in $U$, so
the same bound applies on each piece. By the triangle inequality,
\[
 \|\widetilde\theta(s)-\widetilde\theta(0)\|_\infty
 \leq D_0\lceil s\rceil.
\]
For the lift \eqref{eq:liftedorbit}, the endpoints are $u$ and $v+k$.
Since $\|u-v\|_\infty\leq1$, it follows that
\[
 \|k\|_\infty\leq D_0\lceil s\rceil+1
 \leq(D_0+1)(1+s).
\]
For $s=0$, the lift is constant and $u=v+k$, so
$\|k\|_\infty\leq1$. Taking the maximum of $D_0+1$ over the
finitely many charts proves \eqref{eq:winding}.
\end{proof}

The uniformity comes from the one family \eqref{eq:circlefamily}
defined on all of $U\times[0,1]$. We have not exhausted $U$ by
compact subsets or taken a limit of constants. Consequently the
argument is unaffected by divergence of $A$, $P$, or $P^{-1}$ near
the boundary.

In the flat model of Section~\ref{sec:flat-winding}, the direct
estimate $\|k\|_\infty\leq s+1$ has this same uniformity despite
the diverging period matrix. The short-circle-path argument replaces
that explicit estimate on a general family.

\subsection{Uniform multiplicity for a prescribed winding vector}
\label{sec:multiplicity}
There are now only $O(T^n)$ possible labels. We must still control
how many geodesics can have the same label. The distinction between
parameters and unknowns is central to this step.

\subsubsection{Write one family of equations with constructible coefficients}
Restore the chart index $\ell$ and write
$Z_\ell(b,u)=\tau_\ell(b,\rho(u))$. For counting, use
$u,v\in Q_0=[0,1)^n$: every point of $\T^n$ then has exactly one
representative. The functions $A_\ell$ are still defined on the
closed cube $Q$, as required for their period columns.

The \emph{parameters} of the family are
\[
 (q,y,T,\kappa)\in M\times M\times[1,\infty)\times\R^n.
\]
They specify the base endpoints, a length cutoff, and a real vector
that will later be restricted to integer labels. For fixed parameters,
the \emph{unknowns} are $(b,u,v,s)$: the torus, its initial and final
angular representatives, and the geodesic length. Define the solution
fiber by
\begin{equation}\label{eq:Rfamily}
 \begin{split}
 \cR_\ell(q,y,T,\kappa)=\bigl\{(b,u,v,s):\;&b\in B_\ell,\quad
       u,v\in Q_0,\quad 0<s\leq T,\\
 &\pi Z_\ell(b,u)=q,\quad \pi Z_\ell(b,v)=y,\\
 &A_\ell(b,v)-A_\ell(b,u)+P_\ell(b)\kappa=s e_1\bigr\}.
\end{split}
\end{equation}

\paragraph{Why the two endpoint coordinates have the same $b$.}
The pairs $(b,u)$ and $(b,v)$ describe points of the unit cotangent
bundle $S$, whose dimension is $2n-1$. The points $q,y$ belong to
the $n$-dimensional base manifold and are obtained by projection.
Specifically, for an actual connecting geodesic, let
\[
 z_0=(q,p_0)=Z_\ell(b,u),\qquad
 z_s=\varphi_s(z_0)=(y,p_s)=Z_\ell(b,v).
\]
The momenta $p_0,p_s$ are unit covectors. Conservation of the first
integrals gives
\[
 \frac{d}{dt}F_i(\varphi_t(z_0))
       =\{F_i,E\}(\varphi_t(z_0))=0\quad(2\leq i\leq n),
 \qquad f(z_s)=f(z_0)=b.
\]
The trajectory also stays in the same connected component of the
level set, so its chart index $\ell$ is unchanged. Its angular
position can change from $u$ to $v$ while its torus label $b$ is
fixed. Projecting these two phase-space points gives exactly the
conditions $\pi Z_\ell(b,u)=q$ and $\pi Z_\ell(b,v)=y$ in
\eqref{eq:Rfamily}. They allow $q$ and $y$ to be distinct.

A base point $q$ alone does not specify $b$: one must first choose
an initial unit covector $p_0$, and then $b=f(q,p_0)$. Consequently
different geodesics joining the same $q,y$ may have different values
of $b$. The equation searches over $b$ as an unknown. For a fixed
$b$, it only counts those segments whose initial and terminal
covectors lie on that particular torus; if none exist, the corresponding
solution set is empty. Starting instead with two labels $b,b'$ would
require the additional conservation equation $b'=b$, so
\eqref{eq:Rfamily} has already eliminated that redundant unknown.
In general $p_0\in T_q^*M$ and $p_s\in T_y^*M$ lie in different
cotangent spaces. The conserved equality is $f(z_s)=f(z_0)$;
it does not identify these covectors.
The formula for $b$ in Section~\ref{sec:flat-solutions} illustrates
how this common value can vary between geodesics with the same base
endpoints.

\paragraph{Verification of the analytic hypotheses.}
We now verify the hypotheses of Theorem~\ref{thm:integration}(ii)
for these equations. Regard $q$ and $y$ in the fixed analytic Euclidean
embedding of $M$. The full domain
\[
 \begin{split}
 D_\ell=\bigl\{(q,y,T,\kappa,b,u,v,s):\;&q,y\in M,\ T\geq1,
 \ \kappa\in\R^n,\\
 &b\in B_\ell,\ u,v\in Q_0,\ 0<s\leq T\bigr\}
 \end{split}
\]
is globally subanalytic. On this domain the coordinates of
$\pi Z_\ell(b,u)-q$ and $\pi Z_\ell(b,v)-y$ are globally
subanalytic. Lemmas~\ref{lem:Adefinable} and \ref{lem:developing}
make each coordinate of
\[
 A_\ell(b,v)-A_\ell(b,u)+P_\ell(b)\kappa-s e_1
\]
constructible: evaluating at $(b,u)$ or $(b,v)$ is composition with
a globally subanalytic projection, and multiplication by $\kappa_j$
and finite addition preserve the form \eqref{eq:constructible}.
Thus \eqref{eq:Rfamily} is precisely a finite system of equations
of the type covered by Theorem~\ref{thm:integration}(ii), jointly in
all parameters and unknowns. In particular $T$ and $s$ can be
unbounded, since they enter only through inequalities and linear terms.

For noninteger $\kappa$, the solutions need not represent geodesics,
and a fiber may be infinite. Such parameters are used only to place
every integer-label equation inside the same family of real equations.

\subsubsection{Prove finiteness geometrically, then make it uniform}
\begin{lemma}[Uniform multiplicity for a fixed winding label]
\label{lem:multiplicity}
For each chart $\ell$ there is an integer $m_\ell<\infty$ such that
\begin{equation}\label{eq:multiplicity}
 \#\cR_\ell(q,y,T,k)\leq m_\ell
 \qquad((q,y)\in\cG,\ T\geq1,\ k\in\Z^n).
\end{equation}
\end{lemma}
\begin{proof}
\emph{First show that the relevant fibers are finite.}
Fix $(q,y)\in\cG$, $T\geq1$, and $k\in\Z^n$. By
Corollary~\ref{cor:endpoint}, every solution of \eqref{eq:Rfamily}
gives an actual connecting geodesic with initial vector
\begin{equation}\label{eq:injective}
 (b,u,v,s)\longmapsto\xi=s\,Z_\ell(b,u)^\sharp_g\in T_qM.
\end{equation}
Its length is $s\leq T$ and its terminal point is $y$.

This assignment is injective. Indeed, a nonzero initial vector $\xi$
determines $s=|\xi|_g$ and the unit initial covector
$z=(\xi/s)^\flat_g$, where $\flat_g$ is the inverse of $\sharp_g$.
It then determines $b=f(z)$, the terminal covector $\varphi_s(z)$,
and their unique representatives $u,v\in Q_0$ under $\tau_\ell$.
Conversely, every connecting geodesic whose initial covector is in
$U_\ell$ supplies these unknowns and exactly one integer $k$ by
Corollary~\ref{cor:endpoint}. Thus no geodesic is counted twice
either within a fiber or under different winding labels.

Lemma~\ref{lem:goodpairs} gives finitely many connecting geodesics
of length at most $T$ for the chosen endpoint pair. The injectivity
of \eqref{eq:injective} makes $\cR_\ell(q,y,T,k)$ finite. This
finiteness comes from the compact velocity disk and nonconjugacy,
not from compactness of the half-open angular coordinate domain.

\emph{Now apply uniform finiteness to the real-parameter equations.}
Apply Theorem~\ref{thm:integration}(ii) to the system verified above,
with parameters $(q,y,T,\kappa)$ and solution variables $(b,u,v,s)$.
It bounds every finite solution set by one integer $m_\ell$.
At $(q,y)\in\cG$ and $\kappa=k\in\Z^n$ these sets have just
been proved finite, so the same bound yields \eqref{eq:multiplicity}.
Other parameter values may give infinite solution sets. There is also
no need for $\cG$ or the selected integer parameters to have a
subanalytic description: the bound was established on the whole
real-parameter family before making that selection.
\end{proof}

The order of these two steps prevents a circular argument. Sard's
theorem and compactness establish ordinary finiteness before
Theorem~\ref{thm:integration}(ii) supplies a common bound. Neither the endpoints
nor the cutoff $T$ occur in the resulting constant $m_\ell$.

Section~\ref{sec:flat-solutions} solves the flat system explicitly:
a winding label determines $s$ and $p$, and therefore $b$ and the
family. There $m_+=m_-=1$. Uniform finiteness replaces this explicit
uniqueness in the general case.

\subsection{Sum over the winding labels}\label{sec:sum}
\begin{proof}[Proof of Theorem~\ref{thm:count}]
Fix $(q,y)\in\cG$ and $T\geq1$. Lemma~\ref{lem:goodpairs} places
every connecting geodesic in a retained chart, and disjointness of
the $U_\ell$ makes that chart unique. The preceding section gives
a unique solution of \eqref{eq:Rfamily} and a unique integer label.
By Lemma~\ref{lem:winding}, every such label satisfies
$\|k\|_\infty\leq C_0(1+T)$. Hence
\begin{align}
 N_{q,y}(T)
 &=\sum_{\ell=1}^L
   \sum_{\substack{k\in\Z^n\\\|k\|_\infty\leq C_0(1+T)}}
       \#\cR_\ell(q,y,T,k)\notag\\
 &\leq\left(\sum_{\ell=1}^L m_\ell\right)
       \bigl(2\lfloor C_0(1+T)\rfloor+1\bigr)^n\notag\\
 &\leq C(1+T)^n,\label{eq:finalcount}
\end{align}
where, for example,
$C=\max\{1,(\sum_\ell m_\ell)(2C_0+1)^n\}$ suffices.
The set $\cG$, the winding constant $C_0$, and all multiplicity
constants $m_\ell$ were chosen independently of $q,y,T$.
\end{proof}

\section{From polynomial counting to zero entropy}\label{sec:entropy}
Before integrating the count, we verify its measurability at every
endpoint pair, including those where it may be infinite. This step
uses only a smooth metric on a closed manifold.

\begin{lemma}[Measurability of geodesic counting]\label{lem:measurable}
For each fixed $T>0$, the extended nonnegative function
$(q,y)\mapsto N_{q,y}(T)$ is Borel measurable.
\end{lemma}
\begin{proof}
Let $D_T=\{(q,\xi)\in TM:|\xi|_g\leq T\}$ be the compact disk
bundle, and choose a compatible metric $d_D$ on it. For integers
$r,m\geq1$, let $A_{r,m}\subset D_T^r$ consist of tuples
$(\eta_1,\ldots,\eta_r)$ with
\[
 \Exp(\eta_1)=\cdots=\Exp(\eta_r),\qquad
 |\eta_i|_g\geq1/m,\qquad
 d_D(\eta_i,\eta_j)\geq1/m\quad(i\neq j).
\]
Here $|(q,\xi)|_g$ means the fiber norm $|\xi|_g$.
These are closed conditions, so $A_{r,m}$ is compact. Its image
$K_{r,m}\subset M\times M$ under the common endpoint map is also
compact. Any $r$ distinct nonzero initial vectors in one fiber of
$\Exp$ satisfy the displayed lower bounds for some sufficiently
large $m$. Consequently
\[
 \{(q,y):N_{q,y}(T)\geq r\}
    =\bigcup_{m=1}^{\infty}K_{r,m}.
\]
Every such superlevel set is Borel, proving the assertion, including
at pairs with infinitely many connecting geodesics.
\end{proof}

Ma\~n\'e's formula \cite{Mane,BurnsP} states that
\begin{equation}\label{eq:mane}
 \htop(g)=\lim_{T\to\infty}\frac1T
    \log\int_{M\times M}N_{q,y}(T)\dd\vol_g(q)\dd\vol_g(y).
\end{equation}
The metric identification of unit tangent and cotangent bundles
conjugates their geodesic flows, so this formula applies to our
normalization. Omitting constant geodesics changes the count only
on the diagonal, a null subset of $M\times M$.

\begin{samepage}
\begin{lemma}[Normalization of the integrability hypotheses]
\label{lem:normalization}
Let $(M^n,g)$ be closed, connected, and real analytic, with $n\geq2$.
Each of the following hypotheses supplies integrals satisfying
Definition~\ref{def:integrability} for the same metric $g$.
\begin{enumerate}[label=(\roman*)]
\item There are single-valued real-analytic functions
$I_2,\ldots,I_n:S\to\R$, constant along the unit-speed geodesic
flow and with linearly independent differentials on a dense subset
of $S$. Their radial extensions
\[
 J_i(q,p)=I_i(q,p/|p|_g),\qquad p\neq0,
\]
satisfy $\{J_i,J_j\}|_S=0$.
\item On an open neighborhood $W$ of an entire surface
$E^{-1}(e_*)$, $e_*>0$, there are single-valued real-analytic
functions $F_1=E,F_2,\ldots,F_n$ that Poisson commute on $W$
and have independent ambient differentials on a dense subset of $W$.
\end{enumerate}
\end{lemma}
\end{samepage}
\begin{proof}
\emph{Integrals given on the unit energy surface.}
In (i), each $J_i$ is analytic on $T^*M\setminus0_M$, where
$0_M$ denotes the zero section, and is homogeneous of degree zero
in the fibers. For $c>0$, put $\delta_c(q,p)=(q,cp)$. The identity
$\delta_c^*\omega=c\omega$ gives
\begin{equation}\label{eq:dilationbracket}
 \{F\circ\delta_c,G\circ\delta_c\}
       =c\{F,G\}\circ\delta_c.
\end{equation}
Since $J_i\circ\delta_c=J_i$ and $E\circ\delta_c=c^2E$,
\[
 \{J_i,J_j\}\circ\delta_c=c^{-1}\{J_i,J_j\},\qquad
 \{E,J_i\}\circ\delta_c=c\{E,J_i\}.
\]
The first bracket vanishes on $S$ by assumption. The second vanishes
there because $J_i|_S=I_i$ is a first integral. Every nonzero covector
is a positive dilation of a point of $S$, so both brackets vanish on
all of $T^*M\setminus0_M$. On $S$, independence of the $dI_i$
is equivalent to ambient independence of $dE,dJ_2,\ldots,dJ_n$,
because $dE\neq0$ and $\ker dE=TS$. Thus $E,J_2,\ldots,J_n$
satisfy Definition~\ref{def:integrability}.

Conversely, start with that definition and set $I_i=F_i|_S$.
Let $r(q,p)=(q,p/|p|_g)$ be the radial retraction onto $S$, and
let $\mathcal V=p\partial_p$ be the fiber radial vector field.
On $S$, where $dE(\mathcal V)=1$, differentiating the retraction gives
\[
 dr_z(w)=w-dE_z(w)\mathcal V_z,\qquad
 dJ_i=dF_i-dF_i(\mathcal V)dE.
\]
Pairing these covectors with the Poisson tensor shows
$\{J_i,J_j\}|_S=0$, because every resulting term contains either
$\{F_i,F_j\}$ or a bracket with $E$.
This verifies the equivalence with the unit-energy formulation (i).

\emph{Independence assumed in an ambient neighborhood.}
In (ii), compactness of $E^{-1}(e_*)$ gives numbers
$0<a<e_*<b$ such that $\{a<E<b\}\subset W$. Choose a point
$z_0$ of independence in this annulus and set $e_0=E(z_0)$.
The sphere bundle $S_{e_0}=E^{-1}(e_0)$ is connected: its base is
connected and its fiber is $S^{n-1}$. The ambient analytic section
\[
 \left.(dF_1\wedge\cdots\wedge dF_n)\right|_{S_{e_0}}
 \quad\text{of}\quad
 \left.\bigwedge\nolimits^n T^*(T^*M)\right|_{S_{e_0}}
\]
is nonzero at $z_0$. The identity principle for analytic sections on
a connected analytic manifold makes its nonvanishing locus dense
in $S_{e_0}$. Restriction here means restriction of an ambient
bundle section; pulling it back as a differential form on $S_{e_0}$
would give zero because $dE$ annihilates its tangent space.

Set $c=\sqrt{2e_0}$. Then $\delta_c(S)=S_{e_0}$, and on the open
neighborhood $\delta_c^{-1}(W)$ of $S$ the functions
\[
 \widetilde F_1=c^{-2}F_1\circ\delta_c=E,\qquad
 \widetilde F_i=F_i\circ\delta_c\quad(2\leq i\leq n)
\]
commute by \eqref{eq:dilationbracket}. Invertibility of $d\delta_c$
transfers dense ambient independence to $S$. These functions satisfy
Definition~\ref{def:integrability}; the metric and its unit-speed
geodesic flow have not been rescaled.
\end{proof}

An energy selection can be necessary in (ii) if one retains the
given integrals. On a flat two-torus,
$F_1=E$ and $F_2=(E-\tfrac12)p_1$ commute and are independent
on a dense ambient set, since
\[
 dE\wedge dF_2=(E-\tfrac12)dE\wedge dp_1.
\]
Their differentials are dependent at every point of $S$.
Ambient independence alone therefore does not permit direct use
of the counting theorem on a prescribed energy surface with the
original integrals.

\begin{proof}[Proof of Theorem~\ref{thm:entropy}]
Under Definition~\ref{def:integrability}, Lemma~\ref{lem:measurable}
shows that the count is a nonnegative measurable function. Its integral over
$M\times M$ equals its integral over the full-measure set $\cG$;
even infinite values on the null complement do not affect this
Lebesgue integral. Theorem~\ref{thm:count} therefore gives
\begin{equation}\label{eq:integralbound}
 \int_{M\times M}N_{q,y}(T)\dd\vol_g(q)\dd\vol_g(y)
 \leq C\vol_g(M)^2(1+T)^n\qquad(T\geq1).
\end{equation}
For $T\geq\diam(M)$, every distinct pair is joined by a minimizing
geodesic of length at most $T$. Thus the integral is at least
$\vol_g(M)^2>0$. The logarithms of these lower and upper bounds,
divided by $T$, both tend to zero. Equation~\eqref{eq:mane} yields
$\htop(g)=0$.

Lemma~\ref{lem:normalization} proves the other formulations stated
after the theorem.
\end{proof}

The constant independent of the endpoints is what permits the
average bound \eqref{eq:integralbound}. The fixed set $\cG$, valid
for every $T$, will also let us choose one endpoint pair for all
length cutoffs in Section~\ref{sec:elliptic}. The entropy argument
estimates the full endpoint average in \eqref{eq:mane}; it does not
assume that an invariant null set has zero entropy.

For completeness, a connected closed one-dimensional manifold has
periodic unit cotangent flow on two circles. For a disconnected compact
manifold, assume the stated integrability hypothesis on each connected
component of dimension at least two. There are finitely many components,
and the entropy of their disjoint union is their maximum, again zero.

\section{A further consequence: rational ellipticity}\label{sec:elliptic}
This consequence uses the counting theorem directly. Work with the
space of $H^1$ paths $\gamma:[0,1]\to M$ from fixed $q$ to fixed $y$.
It is homotopy equivalent to $\Omega M$: concatenate with a fixed
path from $y$ back to $q$. The energy functional is
\[
 \mathcal E(\gamma)=\tfrac12\int_0^1|\dot\gamma(t)|_g^2\dd t.
\]
Its critical points are constant-speed geodesics; at a critical point
of length $L$ the energy is $L^2/2$. The kernel of the Hessian consists
of Jacobi fields vanishing at both endpoints, so endpoint nonconjugacy
is precisely nondegeneracy of the critical point. Gromov's estimate
controls the lengths needed to represent loop-space homology, and
Morse inequalities yield the following count;
see \cite{Gromov,Milnor,P92,P94,Pbook}.

We use the lower bound stated in \cite[\S1.4]{Gromov}, with its
length-control argument in \S4.1: for a fixed such pair $q,y$, there
are $a>0$, $b\geq0$, and $T_0$ such that
\begin{equation}\label{eq:gromov}
 \sum_{j=1}^{\lfloor aT-b\rfloor}\dim_{\Q}H_j(\Omega M;\Q)
 \leq N_{q,y}(T)\qquad(T\geq T_0).
\end{equation}
Only constants for this fixed pair are needed. The published counting
statement uses a strict length cutoff and positive homological degrees.
Replacing the strict cutoff by $\leq T$ preserves the inequality;
degree zero will be added separately.

\begin{proof}[Proof of Corollary~\ref{cor:elliptic}]
Choose one $(q,y)\in\cG$. Lemma~\ref{lem:goodpairs} verifies
nonconjugacy for every connecting geodesic, as required in
\eqref{eq:gromov}. For an integer $N\geq1$, put
\[
 T_N=\max\left\{1,T_0,\frac{N+b+1}{a}\right\}.
\]
Then $\lfloor aT_N-b\rfloor\geq N$ and
$T_N\leq C_2(N+1)$ for some $C_2$ independent of $N$.
Simple connectivity makes $\Omega M$ connected, so
$\dim_{\Q}H_0(\Omega M;\Q)=1$. Equations~\eqref{eq:gromov}
and \eqref{eq:countbound} give
\begin{equation}\label{eq:betaproof}
 \beta_M(N)\leq1+N_{q,y}(T_N)
       \leq1+C(1+T_N)^n\leq C'(1+N)^n.
\end{equation}
Increasing $C'$ includes $N=0$.

To deduce rational ellipticity, note that a closed smooth manifold
has the homotopy type of a finite CW complex. Thus our simply connected
$M$ has finite rational type and finite rational
Lusternik--Schnirelmann category, as required by the homotopy growth
theorem. If its rational homotopy were infinite dimensional,
\cite[Theorem 2.1(ii)]{FHT82} would give $c>1$ and degrees
$q_i\to\infty$ with
\begin{equation}\label{eq:hyperbolic}
 \dim_{\Q}(\pi_{q_i}(M)\otimes\Q)\geq c^{q_i}.
\end{equation}
The Milnor--Moore identification, recalled in
\cite[p.~2521]{FHT17}, is
\[
 H_*(\Omega M;\Q)\cong U\bigl(\pi_*(\Omega M)\otimes\Q\bigr),
\]
where $U$ is the universal enveloping algebra of the graded homotopy
Lie algebra. The graded Poincar\'e--Birkhoff--Witt theorem makes its
canonical map into $U$ injective. Hence, for $j\geq2$,
\[
 \dim_{\Q}(\pi_j(M)\otimes\Q)
 \leq\dim_{\Q}H_{j-1}(\Omega M;\Q)\leq\beta_M(j-1).
\]
This contradicts \eqref{eq:hyperbolic} and the polynomial bound
\eqref{eq:betaproof}. Thus the rational homotopy of $M$ is finite
dimensional.
\end{proof}

\end{document}